\documentclass[12pt]{article}
\usepackage[english]{babel}
\usepackage{amssymb,amsmath,amsthm}
\usepackage{MnSymbol}
\newtheorem{theorem}{Theorem}[section]
\newtheorem{lemma}[theorem]{Lemma}
\newtheorem{proposition}[theorem]{Proposition}
\newtheorem{corollary}[theorem]{Corollary}

\newtheorem{example}[theorem]{Example}

{\theoremstyle{definition}}
{\theoremstyle{definition}\newtheorem{definition}[theorem]{Definition}}
{\theoremstyle{definition}}

\newtheorem*{thmgs}{Theorem GS}

\numberwithin{equation}{section}

\def\K{{\mathbb K}}

\def\epsilon{\varepsilon}
\def\kappa{\varkappa}
\def\phi{\varphi}
\def\leq{\leqslant}
\def\geq{\geqslant}

\def\dim{{\rm dim}\,}

\def\ker{\hbox{\tt ker}\,}

\def\deg{{\rm deg}\,}
\def\kxy{\langle x,y \rangle}

\title{Golod-Shafarevich type theorems and  potential algebras}

\author{Natalia Iyudu and Agata Smoktunowicz
}

\date{}

\begin{document}

\maketitle

\begin{abstract}
Potential algebras feature in the minimal model program and noncommutative
resolution of singularities, and the  important cases are when they are finite dimensional, or of linear growth. We develop  techniques, involving Gr\"obner basis theory and generalized Golod-Shafarevich type theorems for potential algebras, to determine finiteness conditions in terms of the potential.

 We consider two-generated potential algebras. Using Gr\"obner bases techniques  and arguing in terms of associated truncated algebra we prove that they cannot have dimension smaller than $8$. This answers a question of Wemyss \cite{Wemyss}, related to the geometric argument of Toda \cite{T}. We derive from the improved version of the Golod-Shafarevich theorem, that if the potential has only terms of degree 5 or higher, then the potential algebra is infinite dimensional. We prove, that potential algebra for any homogeneous potential of degree $n\geq 3$ is infinite dimensional. The proof includes a complete classification of all potentials of degree 3. Then we introduce a certain version of Koszul complex, and  prove that in the class ${\cal P}_n$ of potential algebras with homogeneous potential of degree $n+1\geq 4$, the minimal Hilbert series  is $H_n=\frac{1}{1-2t+2t^n-t^{n+1}}$, so they are all infinite dimensional. Moreover, growth could be polynomial (but non-linear) for the potential of degree 4, and is always exponential for potential of degree starting from 5.

 For one particular type of potential we prove a conjecture by Wemyss, which relates the difference of dimensions of potential algebra and its abelianization with  Gopakumar-Vafa invariants.

\end{abstract}



\section{Introduction}\label{Int}

Questions we study in this paper arise from the fact that potential algebras, appearing in minimal model program, in noncommutative resolution of singularities, such as  contraction algebra introduced
by Donovan-Wemyss \cite{WD}, are important, when they are finite dimensional, or have linear growth.
So it was our goal  to develop techniques allowing to recognize when a potential gives rise to an algebra,
which has this kind of finiteness properties, or extract more information on the algebra,
such as its dimension, in terms of potential. Potential algebras and their versions appear in
many different and related contexts in physics and mathematics and are known also under the names  vacualgebra,
Jacobi algebra, etc. (see, for example, \cite{Bo,BW,DV1,DV2,W}).

 Throughout the paper we use the following notation:  ${\mathbb K}\langle x,y\rangle$ is
the free associative algebra in two variables, $F\in{\mathbb K}\langle x,y\rangle$ is a cyclicly invariant
polynomial, not necessarily homogeneous, however the case of a homogeneous $F$ will be treated separately.
We always assume that $F$ starts in degree $\geq3$, that is, the first three homogeneous
components of $F$ are zero: $F_0=F_1=F_2=0$, which means we suppose generators of $A$ are
linearly independent. We consider the {\it potential algebra} $A_F$, given by two relations,
which are partial derivatives of $F$, i.e. $A_F$ is the
factor of ${\mathbb K}\langle x,y\rangle$ by the ideal $I_F$ generated by
$\frac{\partial F}{\partial x}$ and $\frac{\partial F}{\partial y}$, where the linear maps
$\frac{\partial}{\partial x}:{\mathbb K}\langle x,y\rangle\to {\mathbb K}\langle x,y\rangle$ and
$\frac{\partial}{\partial y}:{\mathbb K}\langle x,y\rangle\to {\mathbb K}\langle x,y\rangle$ are defined on
monomials as follows:

$$
\frac{\partial w}{\partial x}=\left\{\begin{array}{ll}u&\text{if $w=xu$,}\\ 0&\text{otherwise,}\end{array}\right.
\qquad \frac{\partial w}{\partial y}=\left\{\begin{array}{ll}u&\text{if $w=yu$,}\\ 0&\text{otherwise.}\end{array}\right.
$$

 This notion of  noncommutative derivation of free associative algebra was introduced by Kontsevich in \cite{Ko}; an equivalent definition is given, for example, in \cite{G}.  In the context of  free groups, Lie algebras and enveloping algebras of Lie algebras the notion of noncommutative  derivation  has been  investigated since the late  1940's by several authors.   In \cite{fox}, Ralph Fox introduced derivatives over free groups to study invariants of
group presentations (see [9]).  In \cite{U} Umirbaev introduced derivatives for enveloping algebras of Lie algebras, and
 in \cite{Wor} Woronowicz introduced differential calculus on quantum groups (see \cite{khar} for more details).

 The common notation we accept further is $u \rcirclearrowleft$, which means the sum of all cyclic permutations of the monomial $u \in K \langle X \rangle$. It is essential for the definition of derivatives, as we gave them, to consider cyclicly invariant polynomials. For an arbitrary polynomial, the equivalent version of derivatives would sound like in \cite{G}.

Using the improved version of the Golod--Shafarevich theorem,
which takes into account the additional information that the relations arise from a potential, we derive the following fact.

\begin{theorem}\label{deg5}
Let $A_F$ be a potential algebra given by a not necessarily homogeneous potential $F$
having only terms of degree $5$ or higher. Then $A_F$ is infinite dimensional.
\end{theorem}

We prove this theorem in Section~\ref{GS}. We also show that
Theorem 3.1 does not follow from the Golod--Shafarevich theorem~\cite{GSh} applied for algebras given by the same number of relations of the same degrees as the potential algebra $A_F$.
The only facts which could be deduced directly from  the Golod--Shafarevich theorem are  Proposition~\ref{pro2} and Proposition~\ref{pro1} for not necessarily homogeneous potential. They are weaker than our results, therefore the assumption that  $A_F$ is the potential algebra is necessary.

%We prove this theorem in Section~\ref{GS}. We also show that if we would apply
%classical version of the Golod--Shafarevich theorem~\cite{GSh} not involving fully
%the fact that algebra is potential, but only the information on the number and degrees %of relations, we would get the following statement. To ensure infinite dimensionality %the degree of homogeneous potential should be 6 or higher (Proposition~\ref{pro2}), and %for not necessarily homogeneous potential, 7 or higher (Proposition~\ref{pro1}). Hence %it proved to be essential to use  the Golod-Shafarevich theorem together with the fact %of existence of a potential.

In Section~\ref{Hom} we deal first with the case of homogeneous potentials of degree $3$. We classify
all of them up to  isomorphism and see that the corresponding algebras are infinite dimensional. We also compute
the Hilbert series for each of them.

Next, we prove the following theorem.

\begin{theorem}\label{deg4} If $F\in\K\langle x,y\rangle$ is a homogeneous potential of degree
$n\geq 4$, then the potential algebra $A_F$
%$=\K\langle x,y\rangle/{\rm Id}(\frac{\partial F}{\partial x},\frac{\partial %F}{\partial y})$
is infinite dimensional.
Moreover, the minimal Hilbert series in the class ${\cal P}_n$ of
potential algebras with homogeneous potential of degree $n+1\geq 4$ is $H_n=\frac{1}{1-2t+2t^n-t^{n+1}}$.
\end{theorem}

\begin{corollary} Growth of a potential algebra with homogeneous potential of degree $4$
can be polynomial (at least quadratic), but starting from degree $5$ it is always exponential.
\end{corollary}

As a consequence of Example~\ref{ex1}, in the case of potential of degree 4 we have that the algebra $A_{(3)}=A_{F_{(3)}}$  given by the potential $x^2y^2 \rcirclearrowleft$ (that is by the relations $A_{(3)}=\kxy / \{xy^2+y^2x, x^2y+yx^2 \}$) has a minimal Hilbert series, namely $H_{(3)}=\frac{1}{1-2t+2t^3-t^4}.$ It  has polynomial growth of degree not higher than three by reasons which are obvious if one notices that $\frac{1}{1-2t+2t^3-t^4}=\frac{1}{(1+t)(1-t)^3}$, but exact calculations of the terms $a_n$ of the series $H_{(3)}=\sum a_n t^n$ via the recurrence $a_n=2a_{n-1}-2a_{n-3}-a_{n-4}$ shows that linear growth is impossible.

The above theorem, together with classification of potentials of degree $3$,
%two facts together
will ensure that potential algebras with homogeneous potential of degree $\geq 3$ are
always infinite dimensional.
As a tool for the proof of  Theorem ~\ref{deg4} we construct a complex, in a way analogous to the Koszul
complex. However, not all maps in our complex have degree one. One of the maps has degree $n-2$, where $n$ is
the degree of the potential.

%This complex makes sense even if the potential is nonhomogeneous.

 In Section~\ref{Dim}, we  answer a question of  Wemyss, which was motivated by the recent investigation  of $A_{con}$.
For that we use Gr\"obner basis technique
 and arguments involving truncated algebra $A^{(n)}=A/{\rm span}\{u_n\}$, where $u_n$ are monomials of degree bigger than $n$.
It was shown by Michael Wemyss, that  the completion of a potential algebra can have dimension $8$ and he conjectured that this is the possible minimal dimension. We show that his conjecture is true.

\begin{theorem} Let $A_F$ be a potential algebra given by a potential $F$
having only terms of degree $3$ or higher.
 The minimal dimension of  $A_{F}$  is at least $8$. Moreover,
 the minimal dimension of the completion of $A_{F}$  is $8$.
\end{theorem}

We  recall some basic information about $A_{con}$ and about the  completion  of an algebra  in Section ~\ref{Dim2}.
%In Section 5 we consider the conjecture due  to Wemyss,
%%particular cases of which are proved in\cite{WD}.
%The conjecture says that  the difference between the dimension of a potential %algebra and its abelianization
%is a linear combination of squares of natural numbers starting from 2, with %non-negative integer coefficients.
In Section 6   we consider the conjecture formulated by Wemyss and Donovan in
% particular cases of which are proved in
\cite{WD}.
The conjecture says that  the difference between the dimension of a potential algebra and its abelianization is a linear combination of squares of natural numbers starting from 2, with non-negative integer coefficients.
Moreover, in \cite{T} Toda prove that in geometric case these integer coefficients do coincide with Gopakumar-Vafa invariants \cite{Ka}.
We give an example of solution of the conjecture
for one particular type of potential, namely for the potential
$F=x^2y+xyx+yx^2+xy^2+yxy+y^2x+a(y)=x^2y+y^2x+a(y) \rcirclearrowleft,$ where $a=\sum_{j=3}^n a_jy^j\in\K[y]$ is of degree $n>3$ and
has only terms of degree $\geq3$.

\section{Estimates from the Golod-Shafarevich theorem}\label{GS}

In this section we get the following  estimate: if a potential $F$ has only terms of degree 5 or higher, then $A_F$
is infinite-dimensional. We obtain it by applying an improved version
of the Golod-Shafarevich theorem, for not necessarily homogeneous algebras \cite{Z,V},
and additionally incorporating the fact that relations arise from a the potential.

We start by showing, for comparison purposes, that straightforward application of classical version of
the Golod-Shafarevich theorem gives infinite dimensionality of algebra  for not necessarily homogeneous case,
for potentials, having only terms of degree 7 or higher, and for homogeneous potentials of degree $\geq 6$.

First, we recall the Golod--Shafarevich theorem.

\begin{thmgs} Let $A={\mathbb K}\langle x_,\dots,x_d\rangle/{\rm Id}(g_1,g_2,\dots)$, where
each $g_j$ is homogeneous of degree $\geq2$ and assume that non-negative integers $s_2,s_3,\dots$
are such that for each $k\geq2$, the number of the relations $g_j$ of degree $k$ does not exceed $s_k$.
Then the Hilbert series $H_A$ of $A$ satisfies the following lower estimate:
$$
H_A\geq \left|\frac{1}{1-dt+s_2t^2+s_3t^3+\dots}\right|,
$$
where the order on power series is coefficient-wise:
$H=\sum h_jt^j\geq G=\sum g_jt^j$ if $h_j\geq g_j$ for all $j$, and $|H|$ is the series obtained
from $H$ by replacing with $0$th all coefficients starting from the first negative one.
(If all coefficients of $H$ are non-negative, then obviously $|H|=H$).
\end{thmgs}

\begin{proposition}\label{pro1}
Let $F$ be a (not necessarily homogeneous) potential starting with degree $n+1$ with $n\geq 6$
(that is, $F_j=0$ for $j\leq n$). Then $A_F$ is infinite dimensional.
\end{proposition}

\begin{proof} Consider the algebra $\widehat A$ given by the generators $x,y$ and the relations
being all homogeneous components of the relations $\frac{\partial F}{\partial x}$ and
$\frac{\partial F}{\partial y}$ of $A=A_F$. Clearly $\widehat A$ is a quotient of $A$ and therefore $A$
is infinite dimensional provided $\widehat A$ is. Clearly, $\widehat A$ satisfies the conditions of the
Golod--Shafarevich Theorem with $k=2$, $s_j=0$ for $j < n$, and $s_j=2$ for $j\geq n$, that is
we have at most two relations of each degree $\geq n$. Thus the theorem yields
$$
H_{\widehat A}\geq \left|\frac{1}{1-2t+2t^n+2t^{n+1}+\dots}\right|=\left|\frac{1-t}{1-3t+2t^2+2t^n}\right|.
$$

One can check that all coefficients of the series given by last rational function are positive if $n\geq 6$
and that the said series has negative coefficients if $n\leq 5$. Thus $A$ is infinite dimensional if $n\geq 6$.
\end{proof}

Note that the same estimate follows from Vinberg's generalization \cite{V} of the Golod--Shafarevich theorem.
If $F$ is homogeneous, a slightly better estimate follows. Surprisingly, it is not that much better.

\begin{proposition}\label{pro2}
Let $F$ be a homogeneous potential of degree $n+1$ with $n\geq 5$. Then $A_F$ is infinite dimensional.
\end{proposition}

\begin{proof} Clearly, $A$ satisfies the conditions of the Golod--Shafarevich Theorem with $k=2$, $s_n=2$ and
$s_j=0$ for $j\neq n$ (we have  two relations of degree $n$). Thus the theorem yields
$$
H_{A}\geq \left|\frac{1}{1-2t+2t^n}\right|.
$$

One can check that all coefficients of the series given by the last rational function are positive if $n\geq 5$
and that the said series has negative coefficients if $n\leq 4$. Thus $A$ is infinite dimensional if $n\geq 5$,
that is, for potentials of degree $6$ and higher.
\end{proof}

\begin{theorem}\label{deg50}
Let $A_F$ be a potential algebra given by a not necessarily homogeneous potential $F$ having only terms of
degree $5$ or higher. Then $A_F$ is infinite dimensional.
\end{theorem}

\begin{proof} Recall that $A_F=\K\langle x,y\rangle/I$, where $I$ is the ideal generated by
$G=\frac{\partial F}{\partial x}$ and $H=\frac{\partial F}{\partial y}$. Consider
the algebra $B=\K\langle x,y\rangle/J$, where $J$ is the ideal generated by $G$ and $Hx$.
The series from the Golod-Shafarevich theorem for $B$ is $G_B(t)=1-2t^2+t^4+t^5$ since $J$ is an ideal given by
the relation of minimal degree 4 ($s_4=1$) and one relation of minimal degree 5 ($s_5=1$). We apply the non-homogeneous version of the
Golod--Shafarevich theorem from \cite{Z} (page 1187).  Note that there is $t_0\in (0,1)$ such that
$G_B(t_0)<0$. For instance, one can take $t_0=0.654$. From this it follows that $B$ is not only infinite
dimensional but has exponential growth,  see \cite{E}, Theorem~2.7, p.10 for details.

Next we show that $Ix\subset J$. Indeed, $Ix$ is spanned (as a vector space) by $m_1Hm_2x$ and
$m_1Gm_2x$, where $m_1,m_2$ are monomials from $\K\langle x,y\rangle$. An element of the second type
$m_1Gm_2x$ belongs to $J$ since $G\in J$ and $J$ is an ideal. For elements of the first type, we need
to show that they can be expressed as linear combinations of elements of the second type and elements
of the first type containing $Hx$ (with $m_2$ starting with $x$). Indeed, this will suffice since
$Hx\in J$ and $J$ is an ideal. For this purpose we can use the commutation relation $Hy=yH-xG+Gx$ obtained
from the syzygy $[H,y]+[G,x]=0$ (see Lemma~\ref{le1}). So, we use here the fact that our relations $G$ and $H$ are not arbitrary but
are coming from a potential. After applying this commutation relation repeatedly to $Hm_2x$, we can pull
all $y$ with which $m_2$  might start to the left of $H$ ensuring the presence of $Hx$. Hence $Ix\subset J$.

The last step is the following. We suppose that $A_F=\K\langle x,y\rangle/I$ is finite dimensional.
Then the quotient (of vector spaces) $\K\langle x,y\rangle x/Ix$ is also finite dimensional. %since

%$Ix=I\cap \K\langle x,y\rangle x$ and therefore  $\K\langle x,y\rangle x/Ix\subseteq %\K\langle x,y\rangle /I$.

Then $B \bar x=\K\langle x,y\rangle x/J'$ with $J'=J\cap \K\langle x,y\rangle x$, and $\bar x=x+J$, is also finite dimensional
because $Ix\subseteq J'=J\cap \K\langle x,y\rangle x$. But $B$ can be presented as
$$
B={\rm Alg}(\bar y)+B \bar x+B \bar x \bar y+B \bar x \bar y^2+\dots={\rm Alg}(\bar y)+B \bar x{\rm Alg}(\bar y),
$$
where ${\rm Alg}(\bar y)$ is the subalgebra of $B$ generated by $\bar y$. Since $B \bar x$ is finite dimensional,
it follows that $B$ has linear growth. However, this contradicts  the fact that $B$ has exponential
growth obtained in the first part of the proof.
\end{proof}

\section{Homogeneous potential \label{Hom}}\rm

Here we consider the question on infinite-dimensionality of potential algebras in homogeneous case.
This will be the basis for the non-homogeneous arguments as well.

\begin{theorem}\label{deg3} If $F\in\K\langle x,y\rangle$ is a homogeneous potential of degree $3$, then
the potential algebra
$A_F=\K\langle x,y\rangle/{\rm Id}(\frac{\partial F}{\partial x},\frac{\partial F}{\partial y})$ is
infinite dimensional.
\end{theorem}

\begin{proof}
 Let us note first the following easy fact, which will be further used.
 If we have a potential $F(x,y)$ on variables $x,y$, and a linear change of variables
 $u=ax+by$, $v=cx+dy$ let the potential in new variables $u,v$ is $G(u,v)=F(x,y)$, then $\partial_u G  = a (\partial_x F)_{u,v} + b (\partial_y F)_{u,v}$ and
 $\partial_v G  = c (\partial_x F)_{u,v} + d (\partial_y F)_{u,v}.$
 Thus the linear space spanned on derivatives is preserved under the change of variables.

 Since $F$ is a homogeneous cyclicly invariant polynomial of degree $3$, we have
$$
F=ax^3+b(x^2y+xyx+yx^2)+c(xy^2+yxy+y^2x)+dy^3.
$$

Consider the abelianization $F^{\rm ab}\in\K[x,y]$ of $F$, obtained from $F$ by assuming that $x$ and $y$ commute:
$$
F^{\rm ab}=ax^3+3bx^2y+3cxy^2+dy^3.
$$

As $\K$ is algebraically closed, we can write $F^{\rm ab}$ as a product of three linear forms:
$$
F^{\rm ab}=(\alpha_1x+\beta_1y)(\alpha_2x+\beta_2y)(\alpha_3x+\beta_3y).
$$
If the three forms above are proportional, a linear substitution turns $F^{\rm ab}$ into $x^3$. The same
substitution turns $F$ into $x^3$ as well. If two of the three forms are proportional, while the third is
not proportional to the first two, then a linear substitution turns $F^{\rm ab}$ into $3x^2y$. The same
substitution turns $F$ into $x^2y+xyx+yx^2$. Finally, if no two of the above three linear forms are
proportional, then a linear substitution turns $F^{\rm ab}$ into $3x^2y+3xy^2$. The same substitution
turns $F$ into $x^2y+xyx+yx^2+xy^2+yxy+y^2x$.

That is a linear substitution turns $F$ into either $x^3$ or $x^2y+xyx+yx^2$ or $x^2y+xyx+yx^2+xy^2+yxy+y^2x$.
Thus we can assume that
$F\in\{x^3,x^2y+xyx+yx^2,x^2y+xyx+yx^2+xy^2+yxy+y^2x\}$.

If $F=x^3$, then $A=\K\langle x,y\rangle/{\rm Id}(x^2)$. If $F=x^2y+xyx+yx^2$, then
$A=\K\langle x,y\rangle/{\rm Id}(xy+yx,x^2)$. Finally, if $F=x^2y+xyx+yx^2+xy^2+yxy+y^2x$,
then $A=\K\langle x,y\rangle/{\rm Id}(xy+yx+y^2,x^2+xy+yx)=\K\langle x,y\rangle/{\rm Id}(xy+yx+y^2,x^2-y^2)$.
In each case the given quadratic defining relations form a Gr\"obner basis in the ideal
of relations (with respect to the usual degree lexicographical ordering; we assume $x>y$).
In each case, the algebra is infinite dimensional. It has exponential growth for $F=x^3$ and it has
the Hilbert series $H_A=1+2t+2t^2+2t^3+\dots$ in the other two cases (the normal words are $y^n$ and $y^nx$).
\end{proof}

\begin{theorem}\label{deg40} If $F\in\K\langle x,y\rangle$ is a homogeneous potential of degree $n\geq 4$,
then the potential algebra
$A_F$
%=\K\langle x,y\rangle/{\rm Id}(\frac{\partial F}{\partial x},\frac{\partial F}{\partial y})$
is infinite dimensional.
Moreover, the minimal Hilbert series in the class ${\cal P}_n$ of potential algebras with
homogeneous potential of degree $n+1\geq 4$ is $H_n=\frac{1}{1-2t+2t^n-t^{n+1}}$.
\end{theorem}

  Note, that after the proof of this theorem is completed, so we know that the minimal Hilbert series in the class ${\cal P}_n$ of potential algebras with homogeneous
potential of degree $n+1\geq 4$ is $H_n=\frac{1}{1-2t+2t^n-t^{n+1}}$, it immediately follows that algebras from  ${\cal P}_n$ are all infinite dimensional.

\begin{corollary} Growth of a potential algebra with homogeneous potential of degree $4$ can
be polynomial, but starting from degree $5$ it is always exponential.
\end{corollary}

One of the key instruments in our proof is a potential complex of right $A$-modules, which has already appeared in the classical papers of Artin-Schelter \cite{AS}, section one, and in Artin, Tate, Van den Bergh \cite{ATV}  under the name 'potential complex' in section 2.

\begin{definition}

The potential complex is a complex of right $A$-modules

$$
0\to A\mathop{\longrightarrow}^{d_3}A^2\mathop{\longrightarrow}^{d_2}A^2
\mathop{\longrightarrow}^{d_1}A\mathop{\longrightarrow}^{d_0}\K\to 0,
$$

where $d_0$ is the augmentation map,

$$
d_1(u,v)=xu+yv,\quad d_2(u,v)=\left(\begin{array}{cc}\partial_x\partial_xF&\partial_x\partial_yF\\
\partial_y\partial_xF&\partial_y\partial_yF\end{array}\right)\left(\begin{array}{c}u\\ v\end{array}\right),\quad
d_3(u)=(xu,yu).
$$
\end{definition}

\begin{proof} We will start the proof of the main Theorem~\ref{deg40}
with a number of general observations.

\begin{lemma}\label{le1}
For every $F\in\K\langle x,y\rangle$ such that $F_0=0$, $F=x\frac{\partial F}{\partial x}+y\frac{\partial F}{\partial y}$. Furthermore, the equality $F=\frac{\partial F}{\partial x}x+\frac{\partial F}{\partial y}y$ holds if and only if $F$ is cyclicly invariant. In particular, $[x,\frac{\partial F}{\partial x}]+[y,\frac{\partial F}{\partial y}]=0$ if and only if $F$ is cyclicly invariant.
\end{lemma}

\begin{proof}
Trivial.
\end{proof}

\begin{lemma}\label{le2}
Let $F\in\K\langle x,y\rangle$ be cyclicly invariant such that $F_0=F_1=0$, and $A=\langle x,y\rangle/I$ with
$I={\rm Id}(\partial_xF,\partial_yF)$ be the corresponding potential algebra. %$(\partial_x$, $\partial_y$ stand
%for $\partial/\partial x$ and $\partial/\partial y$, respectively$)$.
Then the potential complex of right A-modules
is exact at the three rightmost terms.

\end{lemma}

\begin{proof}(of Lemma~\ref{le2})

First, we show that $d^2=0$. Obviously, $d_0\circ d_1=0$. Note that this kind of
complex ending is rather common. It is shared, for instance, by the Koszul complex of a quadratic algebra.
Next, we show that $d_1\circ d_2=0$. Indeed,
\begin{align*}
d_1(d_2(a,b))&=d_1(\partial_x\partial_xFa+\partial_x\partial_yFb,\partial_y\partial_xFa+\partial_y\partial_yFb)
\\
&=x(\partial_x\partial_xFa+\partial_x\partial_yFb)+y(\partial_y\partial_xFa+\partial_y\partial_yFb)
\\
&=(x\partial_x\partial_xF+y\partial_y\partial_xF)a+(x\partial_x\partial_yF+y\partial_y\partial_yF)b
\\
&=(\partial_xF)a+(\partial_yF)b=0,
\end{align*}

where the second last equality is due to Lemma~\ref{le1}, while the last equality follows from the definition of $A$.

Now we show that $d_2\circ d_3=0$. Indeed,
$$
d_2(d_3(u))=d_2(xu,yu)=(\partial_x(\partial_xFx+\partial_yFy)u,\partial_y(\partial_xFx+\partial_yFy)u)
=((\partial_xF)u,(\partial_yF)u)=(0,0),
$$
where the second last equality is due to Lemma~\ref{le1} and cyclic invariance of $F$.

Now the exactness of the complex in question at $\K$ and at the rightmost $A$ are obvious. It remains to check its exactness at the rightmost $A^2$. That is, we have to verify that if $d_1(u,v)=0$, then $(u,v)=d_2(a,b)$ for some $a,b\in A$.

Let $u,v\in A$ be such that $d_1(u,v)=0$. Pick $u_1,u_2\in\K\langle x,y\rangle$ such that $u_1+I=u$ and $v_1+I=v$. Since $xu+yv=0$ in $A$, we have  $xu_1+yv_1\in I$. Since $I=xI+yI+\partial_xF\K\langle x,y\rangle+\partial_yF\K\langle x,y\rangle$, we see that $xu_1+yv_1=\partial_xF a_1+\partial_yF b_1+xp+yq$, where $a_1,b_1\in \K\langle x,y\rangle$ and $p,q\in I$. Using Lemma~\ref{le1}, we have $\partial_xF=x\partial_x\partial_xF+y\partial_y\partial_xF$ and $\partial_yF=x\partial_x\partial_yF+y\partial_y\partial_yF$. Plugging these into the previous equality, we get
$$
xu_1+yv_1=(x\partial_x\partial_xF+y\partial_y\partial_xF)a_1+(x\partial_x\partial_yF+y\partial_y\partial_yF)b_1+xp+yq.
$$

Rearranging the terms, we arrive to
$$
x(u_1-p-\partial_x\partial_xFa_1-\partial_x\partial_yFb_1)+y(v_1-q-\partial_y\partial_xFa_1-\partial_y\partial_yFb_1)=0,
$$
where the equality holds in $\K\langle x,y\rangle$. This can only happen if both summands in the above display are zero:
$$
u_1-p-\partial_x\partial_xFa_1-\partial_x\partial_yFb_1=v_1-q-\partial_y\partial_xFa_1-\partial_y\partial_yFb_1=0.
$$
Factoring out $I$ and using that $p,q\in I$, $u_1+I=u$ and $v_1+I=v$, we get
$$
u=\partial_x\partial_xFa+\partial_x\partial_yFb,\quad v=\partial_y\partial_xFa+\partial_y\partial_yFb
$$
in $A$, where $a=a_1+I$ and $b=b_1+I$. That is, $(u,v)=d_2(a,b)$, as required. By this the proof of the lemma is complete.
\end{proof}

The next step in the proof of Theorem~\ref{deg40} will be to construct an example for which the above complex is exact.
%Later we shall show that this is the case for a generic homogeneous potential.

In the Example below we calculate Hilbert series directly, using word combinatorics and Gr\"obner bases techniques.
%Then we derive from this exactness of the potential complex as follows.'
Note that the same example appeared in \cite{Bo}, lemma 5.2(2)
 as an example of a potential algebra with rose quiver for which a complex of bimodules responsible for the Calabi-Yau property is exact.
 %, and the Hilbert series is calculated, as a consequence of exactness.'

%This complex can be obtained from the potential complex by tenzoring by A from the %left. Note that exactness of these two complexes in general is not equivalent

\begin{example}\label{ex1} For $n\geq 3$, consider the homogeneous degree $n+1$ potential
$$
F= {x^{n-1}y^2} \rcirclearrowleft.
$$

%$$
%x^{n-1}y^2+x^{n-2}y^2x+{\dots}+xy^2x^{n-2}+y^2x^{n-1}+yx^{n-1}y.
%$$

Denote the corresponding potential algebra $B$, then the Hilbert series of $B$ is given by
$H_B(t)=\frac1{1-2t+2t^n-t^{n+1}}$ and the complex of Lemma~\ref{le2} for $B$ is exact.
\end{example}

\begin{proof} The defining relations $\partial_xF=x^{n-2}y^2+x^{n-3}y^2x+{\dots}+y^2x^{n-2}$ and
$\partial_yF=x^{n-1}y+yx^{n-1}$ of $B$ form a reduced Gr\"obner basis in the ideal of
relations of $B$ with respect to the left-to-right degree lexicographical ordering assuming $x>y$.
Indeed the leading monomials $x^{n-2}y^2$ and $x^{n-1}y$ of the defining relations have one overlap
only: $x^{n-1}y^2=x(x^{n-2}y^2)=(x^{n-1}y)y$, which resolves. Knowing the Gr\"obner basis,
we can obtain certain information on normal words - those words which does not contain as a subwords leading monomials of the Gr\"obner bases. In our case leading monomials of the Gr\"obner bases are $x^{n-2}y^2$ and $x^{n-1}y, n \geq 3$, and it is convenient to find a recurrence relation on the number of normal words $a_n$ of length $n$, which will give us a presentation of the Hilbert series as a rational function.
This could be obtained from the following combinatorial  formulas.
Denote by $\cal N$ the set of normal words, and let $ B_m$ be the set of normal words of degree $n$ ending with $x$ and by $C_m$ the set of normal words of degree $m$, ending with $y$. Obviously, we have a disjoint union $N_m=B_m \cup C_m$. Then, it is easy to see that

$$B_{m+1}=B_m x \cup C_m x$$

and

$$C_{m+1}=(B_m \setminus (B_{m-n+1} \cup C_{m-n+1}) x^{n-1})y \cup
(C_m \setminus  (C_{m-n+1} x^{n-2}y))y.  $$

Denote by $a_m$ the number of normal words of length $n$, and let $\vert B_m \vert=b_m, \,
\vert C_m \vert=c_m$, then from the above we have:

$$b_{m+1}=b_m+c_m$$

and

$$c_{m+1}=b_m+c_m-2c_{m-n+1}-b_{m-n+1}$$

From this we get a recurrence relation

$$a_{m+1}=2a_m-2a_{m-n+1}+a_{m-n},$$

which together with several few terms of the Hilbert series:
 $a_0=1, a_1=2, ...,a_{n-2}=2^{n-2}, a_{n-1}=2^{n-1}, a_n=2^n-2$,
 will give the required rational expression for the Hilbert series:
%Counting the generating function gives us  the Hilbert series, so we have:
$H_A(t)=\frac1{1-2t+2t^n-t^{n+1}}$.

It remains to show that the complex
$$
0\to B\mathop{\longrightarrow}^{d_3}  B^2 \mathop{\longrightarrow}^{d_2}
B^2\mathop{\longrightarrow}^{d_1} B \mathop{\longrightarrow}^{d_0}\K\to 0
$$
is exact, where $d_0$ is the augmentation map,
$$
d_1(u,v)=xu+yv,\quad d_2(u,v)=\left(\begin{array}{cc}\partial_x\partial_xF&\partial_x\partial_yF\\
\partial_y\partial_xF&\partial_y\partial_yF\end{array}\right)\left(\begin{array}{c}u\\ v\end{array}\right),\quad
d_3(u)=(xu,yu).
$$

Now observe that this complex is exact at the leftmost $B$. Indeed, this exactness is
equivalent to the injectivity of $d_3$. Since none of the two leading monomials of the
elements of the Gr\"obner basis starts with $y$, the set of normal words is closed under multiplication by $y$ on the left. Hence the map $u\mapsto yu$ from $B$ to itself is injective and
therefore $d_3$ is injective. By Lemma~\ref{le2}, the complex is exact at its three rightmost terms.
Thus it remains to verify exactness at the leftmost $B^2$.

Set $b_k=\dim B_k$. Consider the $k^{\rm th}$ slice of the complex:
$$
0\to B_k\to B_{k+1}^2\to B_{k+n}^2\to B_{k+n+1}\to 0.
$$
By exactness at $\K$ and the rightmost $B$, we have $d_1(B_{k+n}^2)=B_{k+n+1}$.
Hence $\dim \ker d_1\cap B_{k+n}^2=2b_{k+n}-b_{k+n+1}$. By exactness at the rightmost
$B^2$, $d_2$ maps $B_{k+1}^2$ onto $\ker d_1\cap B_{k+n}^2$. Hence
$\dim \ker d_2\cap B_{k+1}^2=2b_{k+1}-2b_{k+n}+b_{k+n+1}$. Finally, $d_3$ is
injective and therefore $\dim d_3(B_k)=b_k$. Thus the exactness of the slice is
equivalent to the equality $b_k=2b_{k+1}-2b_{k+n}+b_{k+n+1}$. On the other hand, we know that
$b_k$ are the Taylor coefficients of the rational function $\frac1{1-2t+2t^n-t^{n+1}}$, which are easily
seen to satisfy the recurrent relation $b_{k+n+1}=2b_{k+m}-2b_{k+1}+b_{k}$. Hence all the slices of the
complex are exact and therefore the entire complex for $B$ is exact.
\end{proof}

Denote by ${\cal P}_n$ the class of all potential algebras with homogeneous potential of degree
$n+1.$

In the remaining part of the proof, we will show that the Hilbert series of the algebra $B$ with potential of degree $n+1$ is actually minimal in the class ${\cal P}_n$ ($n\geq 3$), which ensure that any algebra in this class is infinite dimensional.

\begin{proposition}\label{minser} For every $n\geq 3$, the Hilbert series of the potential
algebra $B$ given by the potential ${x^{n-1}y^2} {\rcirclearrowleft}$ is minimal in the
class ${\cal P}_n$ of potential algebras with homogeneous potentials of degree $n+1$ on two generators.
\end{proposition}

\begin{proof} First, note that for every $A\in{\cal P}_n$, the $k^{\rm th}$ coefficient of the
Hilbert series is $2^k$ for each $k<n$, the same as for the free algebra $T=\K\langle x,y\rangle$.
Since $B\in{\cal P}_n$, the coefficients up to degree $n-1$ of $H_{B}$ are indeed minimal.
Now each $A\in{\cal P_n}$ is given by two relations of degree $n$. Then $\dim A_n$ is $2^{n}-2=\dim T_n-2$
if these relations are linearly independent and is greater otherwise. Since the defining relations of
$B$ are linearly independent, $\dim B_n=2^{n}-2$ and is minimal. Consider now
$\dim B_k$ with $k=n+1$.  For an arbitrary $A\in{\cal P_n}$, the component of degree
$n+1$ of the ideal of relations is the linear span of $8$ elements being the two relations $\partial_xF$
and $\partial_yF$ (here $F$ is the potential for $A$) multiplied by the variables
$x$ and $y$ on the left or on the right. However these 8 elements exhibit at least one non-trivial
linear dependence $[\partial_xF,x]+[\partial_yF,y]=0$. Thus $\dim A_{n+1}\geq 2^{n+1}-7$. We
already know the Hilbert series of $B$, which gives $\dim B_{n+1}=2^{n+1}-7$.
So, the $n+1$-st coefficient  of the Hilbert series of $B$ is again minimal.

We proceed in the following way. Assume $k$ is a non-negative integer such that the
coefficients of $H_B$ are minimal up to degree $k+n$ inclusive. We shall verify
that the degree $k+n+1$ coefficient of $H_B$ is minimal as well, which would complete the inductive proof.
The last paragraph was actually providing us with the basis of induction. Consider the slice of the above complex.
$$
0\to A_k\to A_{k+1}^2\to A_{k+n}^2\to A_{k+n+1}\to 0
$$
for algebras $A\in{\cal P_n}$. Note that the coefficients of $H_B$ are minimal up to degree $k+n$.

This means that $\dim B_j=\dim A_j$ for $j\leq k+n$ for Zariski generic $A\in {\cal P}_n$. Indeed,
it is well-known and easy to show that in a variety of graded algebras the set of algebras minimizing the
dimension of any given graded component is Zariski open. Thus generic members of the variety will have
component-wise minimal Hilbert series.
To proceed with the proof, we need the following lemma (which is formulated here in a slightly larger generality than we will need).

%Roughly speaking the lemma says  that
%for any exact complex where maps depend polynomially on parameters (from some variety %with one connected component, for example $K^n$) the rank of matrices defining maps is %maximal.
% Based on this lemma in proposition 3.8 we derive for the potential complex, that the %complex in general position is exact, if there exist at least one exact complex in this %variety.
%In fact, it could be done in the same way for any complex, satisfying conditions of %lemma 3.8., for example for Koszul complex. This statement for Koszul complex is %referred to in \cite{PoPo} as  fact due to Drinfeld, the proof presented in %\cite{PoPo}, section 6.2 is different and based on definition of Koszulity in terms of %distributivity of latices.

\begin{lemma} \label{maxra}
Let $n,m,N,k_1,\dots,k_m$ be positive integers and for $1\leq j\leq m$, $r_j:\K^N\to \K\langle
x_1,\dots,x_n\rangle$ be a polynomial map taking values in degree $k_j$ homogeneous component of
$\K\langle x_1,\dots,x_n\rangle$. For $s=(s_1,...,s_N)\in\K^N$, $A^s$ is the algebra given by
generators $x_1,\dots,x_n$ and relations $r_1(s),\dots,r_m(s)$. Assume also that $\Lambda$ is a
$p\times q$ matrix, whose entries are degree $d$ homogeneous elements of
$\K[s_1,\dots,s_N]\langle x_1,\dots,x_n\rangle$. For every fixed $s$, we can interpret $\Lambda$ as a map
from $(A^s)^q$ to $(A^s)^p$ (treated as free right $A$-modules) acting by multiplication of the matrix $\Lambda$ by a
column vector from $(A^s)^q$.
Fix a non-negative integer $i$ and let $U$ be a non-empty Zariski open subset of $\K^N$ such that $\dim A^s_i$ and
$\dim A^s_{i+d}$ do not depend on $s$ provided $s\in U$. For $s\in\K^N$ let $\rho(i,s)$ be the rank of $\Lambda$ as a
linear map from $(A_i^s)^q$ to $(A^s_{i+d})^p$ and $\rho_{\max}(i)=\max\{\rho(i,s):s\in U\}$.
Then the set $W_i=\{s\in U:\rho(i,s)=\rho_{\max}(i)\}$ is Zariski open in $\K^N$.
\end{lemma}

\begin{proof} Let $t\in W_i$. Then $\rho(i,t)=g$, where $g=\rho_{\max}(i)$.
Pick linear bases of monomials $e_1,\dots,e_u$ and $f_1,\dots,f_v$ is $A^t_i$ and $A^t_{i+d}$ respectively.
Obviously, the same sets of monomials serve as linear bases for $A^s_i$ and $A^s_{i+d}$ respectively.
For $s$ from a Zariski open set $V\subseteq U$. Then $\Lambda$ as a linear map from
$(A_i^s)^q$ to $(A^s_{i+d})^p$ for $s\in V$ has an $u^q\times v^p$ matrix $M_s$ with respect to the said
bases. The entries of this matrix depend on the parameters polynomially. Since the rank of this matrix
for $s=t$ equals $g$, there is a square $g\times g$ submatrix whose determinant is non-zero when
$s=t$. The same determinant is non-zero for a Zariski open subset of $V$. Thus for $s$ from the last set the rank of $M_s$ is at least $g$. By maximality of $g$, the said rank equals $g$.
Thus $t$ is contained in a Zariski open set, for all $s$ from which $\rho(i,s)=g$. That is, $W_i$ is Zariski open.
\end{proof}

We are back to the proof of Proposition~\ref{minser}.
For the sake of brevity, denote $a_j=\min\{\dim A_j:A\in{\cal P}_n\}$. By our assumption, $\dim B_j=a_j$
for all $j\leq k+n$. Let $U=\{A\in{\cal P}_n:\dim A_j=a_j\ \text{for}\ j\leq k+n\}$. Then $B$ belongs
to the Zariski open set $U$ (since ${\cal P}_n$ is just a finite dimensional vector space over $\K$ we can
identify it naturally with some $\K^N$ and speak of Zariski open sets etc.). By Lemma~\ref{maxra}, the rank of
$d_3:A_k\to A_{k+1}^2$ is maximal for a Zariski generic $A\in U$. Obviously, this rank can not exceed $\dim A_k=a_k$.
On the other hand our complex is exact for $A=B$ and therefore $d_3:A_k\to A_{k+1}^2$ is injective and
has rank $\dim B_k=a_k$ for $A=B$. Hence, the set $U_1$ of $A\in U$ for which the rank of
$d_3:A_k\to A_{k+1}^2$ equals $a_k=\dim A_k$ is a non-empty Zariski open subset of $U$. Obviously, $B\in U_1$.
Since for every $A\in U_1$, $d_3:A_k\to A_{k+1}^2$ is injective and $d_3(A_k)$ is contained in the kernel of $d_2$,
the rank of $d_2:A_{k+1}^2\to A_{k+n}^2$ is at most $2a_{k+1}-a_k$. Since the complex is exact for $A=B$, the
same rank for $A=B$ equals $2a_{k+1}-a_k$, so the maximal possible rank for $A\in U_1$ is $2a_{k+1}-a_k$. Let
$U_2$ be the set of $A\in U_1$ such that the rank of $d_2:A_{k+1}^2\to A_{k+n}^2$ equals $2a_{k+1}-a_k$. By
Lemma~\ref{maxra}, $U_2$ is Zariski open. Obviously, $B\in U_2$. Then for $A\in U_1$, $d_2(A_{k+1}^2)$
has dimension $2a_{k-1}-a_k$. Since our complex is exact at the rightmost $A^2$, the dimension of
$(\ker d_1)\cap A_{k+n}^2$ is $2a_{k-1}-a_k$ for each $A\in U_2$. Since our complex is exact at the rightmost $A$,
$d_1(A_{k+n}^2)=A_{k+n+1}$. Hence $\dim A_{k+n+1}=2a_{k+n}-2a_{k+1}+a_k$ for every $A\in V_2$. Since for Zariski
generic $A\in V_2$, $\dim A_{k+n+1}=a_{k+n+1}$ and since $B\in U_2$, we have $\dim B_{k+n+1}=
a_{k+n+1}$, which completes the inductive proof.
\end{proof}

This completes the proof of the Theorem~\ref{deg40}.
\end{proof}

\section{The dimension of a potential algebra cannot be smaller than 8}\label{Dim}

  In this chapter we show that the minimal degree of a potential algebra is at least $8$. This answers a question of Michael Wemyss from \cite{Wemyss}, related to a geometric argument of Toda \cite{T}. In \cite{Wemyss}, Wemyss showed that the algebra with the potential
 $F={1\over 3}(y^{2}x+yxy+xy^{2})+x^{3}+x^{4}$
  has degree $9$. It is not known if there exists  a potential algebra whose dimension is  $8$.

The idea behind  our proof is  to embed our free algebra in a power series ring, and then consider an analog of the Diamond Lemma for power series rings. We use a short proof which is based on this idea  by considering the associated truncated algebra (see the  proof of Theorem ~\ref{important}).

Recall that, as above, $\K\langle x,y \rangle $ is the free associative algebra in $2$ variables, $F\in \langle x,y\rangle $ is a cyclic invariant polynomial.

\begin{lemma}\label{one}
Let $F\in \K\langle x,y\rangle $ be a cyclic invariant polynomial which is a  linear combination of elements of degree $3$ or larger, then
\[ \textstyle \left[x, {\frac {\partial F}{\partial x}}\right]
 =\left[{\frac {\partial F}{\partial y}}, y\right]\]

 Moreover, if $F$ is homogeneous of degree $3$ then elements
 \[ \textstyle x\cdot {\frac {\partial F}{\partial x}}-{\frac {\partial F}{\partial x}}\cdot x, y\cdot {\frac {\partial F}{\partial x}}-{\frac {\partial F}{\partial x}}\cdot y, {\frac {\partial F}{\partial y}}\cdot x-x\cdot {\frac {\partial F}{\partial y}}\]
are linearly dependent over $\K$.
In particular there are $\alpha _{1}, \alpha _{2}, \alpha _{3}\in \K$ (not all zero) such that
$\alpha _{1}\cdot (x\cdot {\frac {\partial F}{\partial x}}-{\frac {\partial F}{\partial x}}\cdot x)+ \alpha _{2}( y\cdot {\frac {\partial F}{\partial x}}-{\frac {\partial F}{\partial x}}\cdot y)+\alpha _{3}\cdot ({\frac {\partial F}{\partial y}}\cdot x-x\cdot {\frac {\partial F}{\partial y}})=0.$
\end{lemma}

\begin{proof}  The first part follows from Lemma ~\ref{le1}.
For the second part, observe that  $F$ is of degree $3$, hence it is a linear combination
of elements $x^{3}, y^{3}, x^{2}y+xyx+yx^{2}, y^{2}x+ yxy+ xy^{2}$.

We can write elements
\[ \textstyle x\cdot {\frac {\partial F}{\partial x}}-{\frac {\partial F}{\partial x}}\cdot x, y\cdot {\frac {\partial F}{\partial x}}-{\frac {\partial F}{\partial x}}\cdot y, {\frac {\partial F}{\partial y}}\cdot x-x\cdot {\frac {\partial F}{\partial y}}\]
for each $F\in \{y^{3}, x^{3}, x^{2}y+xyx+yx^{2}, y^{2}x+yxy+xy^{2}\}$ and observe that in each case these elements are linear
combination of elements \[x^{2}y-yx^{2}, y^{2}x-xy^{2}.\]
\end{proof}

\begin{lemma}\label{two}
Let $F\in \K\langle x, y\rangle $ be a cyclic invariant polynomial which is homogeneous of degree $3$.  Then ${\frac {\partial F}{\partial x}}, {\frac {\partial F}{\partial y}}\in span _{\K}\{x^{2}, y^{2}, xy+yx\}$.
Moreover if ${\frac {\partial F}{\partial x}}$ and ${\frac {\partial F}{\partial y}}$ are linearly independent over $\K$ then the set
\[ \textstyle S=\{ {\frac {\partial F}{\partial x}}\cdot x, {\frac {\partial F}{\partial x}}\cdot y, x\cdot {\frac {\partial F}{\partial x}}, y\cdot {\frac {\partial F}{\partial x}}, {\frac {\partial F}{\partial y}}\cdot x, {\frac {\partial F}{\partial y}}\cdot y, x\cdot {\frac {\partial F}{\partial y}}, y\cdot {\frac {\partial F}{\partial y}}\}\]
spans a vector space over the field $\K$ of dimension at least $7$. Moreover, if the dimension is $7$, then ${\frac {\partial F}{\partial x}}$ and ${\frac {\partial F}{\partial y}}$ form a Gr{\" o}bner Bases.
\end{lemma}

\begin{proof}
Observe that ${\frac {\partial F}{\partial x}}, {\frac {\partial F}{\partial y}} \in { span} _{\K}\{x^{2}, y^{2}, xy+yx\}$ since \[F\in span _{\K}\{x^{3},y^{3},  x^{2}y+xyx+ yx^{2}, y^{2}x+ yxy+ xy^{2}, y^{3}\}.\]
We can introduce the lexicographical ordering on the set  of monomials in $x,y$, with $x>y$.
Notice that the leading monomials of ${\frac {\partial F}{\partial x}}$ and ${\frac {\partial F}{\partial y}}$ are in the set $\{x^{2}, xy, y^{2}\}$ since ${\frac {\partial F}{\partial x}}, {\frac {\partial F}{\partial y}}\in span _{\K}\{x^{2}, y^{2}, xy+yx\}$.
Let $n(1)$ be the leading monomial of ${\frac {\partial F}{\partial x}}$ and $n(2)$ be the leading monomial of ${\frac {\partial F}{\partial y}}$. We have  ${\frac {\partial F}{\partial x}}=n(2)k(1)+g(1)$ and ${\frac {\partial F}{\partial y}}=n(2)k(2)+g(2)$ for some $k(1), k(2)\in \K$ and some $g(1), g(2)\in \K\langle x,y\rangle $.

  Consider monomials of degree $3$ in $\K\langle x,y\rangle $ which don't contain either $n(1)$ nor $n(2)$ as a subword.
Then, there are exactly $2$  such monomials, call them $t(1), t(2)$, since $n(1), n(2)\in \{xx,xy, yy\}.$ This can be shown by considering all the possible cases for $n(1)$ and $n(2)$.

   %(otherwise we can reduce the elements of $Q$ using ${\frac {\partial %F}{\partial x}}=n(2)k(1)+g(1)$ and ${\frac {\partial F}{\partial %y}}=n(2)k(2)+g(2)$).

  Notice that, every monomial of degree $3$ is a linear combination of $t(1)$ and $t(2)$, and elements from the set $S$. The linear space spanned by elements $t(1)$ and $t(2)$ will be denoted $T$.

 Let $Q$ be a linear space such that \[Q\oplus span _{\K}S=A(3)\] where $A(3)$ is the linear space of elements of degree $3$ in $\K\langle x,y\rangle$ (where $x$ and $y$ have the usual gradation $1$).  We can assume that $Q\subseteq T$.

  Suppose that we have applied the Diamond Lemma to relations ${\frac {\partial F}{\partial x}}$ and ${\frac {\partial F}{\partial y}}$ to resolve ambiguities involving $n(1)$ and $n(2)$. If there is some of ambiguity which doesn't resolve (this happens exactly when  ${\frac {\partial F}{\partial x}}$ and ${\frac {\partial F}{\partial y}}$ are not Gr{\" o}bner bases),
then we have a  relation of degree $3$ which  has the leading monomial which doesn't contain neither $n(1)$ nor $n(2)$ as a subword (by construction this relation is in the ideal generated by ${\frac {\partial F}{\partial x}}$ and ${\frac {\partial F}{\partial y}}$).

 Consequently, $Q$ is a proper subspace of the linear space of elements of degree $3$ which don't contain $n(1)$ and $n(2)$ as a subword, therefore $Q$ has dimension smaller than $2$ (recall that $T$ has dimension $2$). It follows that $S$ has dimension larger than $6$.

 Therefore, if ${\frac {\partial F}{\partial x}}$ and ${\frac {\partial F}{\partial y}}$ don't form a Gr{\" o}bner bases then $S$ spans a linear space of dimension at least $7$.

     On the other hand, if ${\frac {\partial F}{\partial x}}$ and ${\frac {\partial F}{\partial y}}$ form a Gr{\" o}bner bases then all ambiguities are resolved, so $Q=T$
by the Diamond Lemma (since our algebra is graded), and so $S$ spans a vector space of dimension exactly $7$.
 \end{proof}

\begin{lemma}\label{three}
Let notation be as in Lemma ~\ref{two}.
Let ${\frac {\partial F}{\partial x}}, {\frac {\partial F}{\partial y}}$ be linearly independent over $\K$, then  $S$ spans a linear space of dimension exactly $7$.  Moreover, ${\frac {\partial F}{\partial x}}, {\frac {\partial F}{\partial y}}$form a Gr{\" o}bner basis.
\end{lemma}

\begin{proof}
Observe that by Lemma ~\ref{one} the dimension of the linear space spanned by $S$ is at most $7$. By Lemma ~\ref{two} the dimension is $7$. The result then follows from Lemma ~\ref{two}.
\end{proof}

 %In the next theorem we will use the following notation.
 %Let $F\in \K\langle x,y\rangle $ be a cyclic invariant polynomial, and $I=I_F$ an % ideal generated by partial derivatives.
  %will denote  the ideal generated by ${\frac {\partial F}{\partial x}}$ and ${\frac {\partial F}{\partial y}}$.
   Denote by $A(i)$ the linear subspace of $\K\langle x,y\rangle $  spanned by monomials of degree $i$.

      \begin{theorem}\label{important}
Let $\K$ be a field. Let $G\in \K\langle x,y\rangle $ be a cyclic invariant polynomial which is a linear combination of monomials of degrees larger than two.
%Let $I$ be the ideal generated by $\frac {\partial G}{\partial x}$ and $\frac {\partial %G}{\partial y}$ in $\K\langle x,y\rangle $,
    Then $A_G=\K\langle x,y\rangle /I_G$ has at least $8$ elements linearly independent over $\K$.
\end{theorem}

\begin{proof} We can write $G=F+H$ where $F\in \K\langle x, y\rangle $ is a cyclic invariant polynomial which is  homogeneous of degree $3$ and $H\in \K\langle x,y\rangle $ is a cyclic invariant polynomial which is a  linear combination of monomials of degrees larger than three.  Let $J$ be the ideal generated by  $\frac {\partial G}{\partial x}$ and $\frac {\partial G}{\partial y}$ and  all monomials of degree $5$. Let  $I$ be the ideal generated by $\frac {\partial F}{\partial x}$ and $\frac {\partial F}{\partial y}$ and all monomials of degree $5$.
Clearly $1,x,y\notin J +A(2)+A(3)+A(4)$ since $\frac {\partial G}{\partial x}$, $\frac {\partial G}{\partial y}$ are linear combination of monomials with degrees larger than $2$. We will consider two cases.

{\bf Case 1.}  Suppose that $\frac {\partial F}{\partial x}$, $\frac {\partial F}{\partial y}$ are linearly independent over $\K$. Notice that there are  $2$ monomials of degree $2$, call then $p(1), p(2)$, such that any nontrivial linear combination of these monomials doesn't belong to $I+A(3)+A(4)$, and hence  doesn't belong to $J+A(3)+A(4)$, since $I+A(3)+A(4)=J+A(3)+A(4)$.

We claim that there are exactly $2$ monomials $m(1), m(2)$ of degree $3$ such that every non-trivial linear combination of $m(1)$ and $m(2)$ is not in $J+A(4)$.
Let $m(1), m(2)$ be monomials of degree $3$ such that   every nontrivial linear combination of $m(1)$ and $m(2)$ is not in $I+A(4)$. By Lemma ~\ref{three} such monomials $m(1), m(2)$ exist.
We will show that this is a good choice of $m(1), m(2)$.  Suppose on the contrary that there is $m$ which is a nontrivial linear combination  of $m(1)$ and $m(2)$ and $m\in J+A(4)$. It follows that $m\in \K\cdot \frac {\partial G}{\partial x}+\K\cdot \frac {\partial G}{\partial y}+S'+A(4)+\sum_{i=5}^{\infty }A(i)$ where
$S'=span _{\K}\{x\cdot \frac {\partial G}{\partial x}$, $x\cdot \frac {\partial G}{\partial y}$, $y\cdot \frac {\partial G}{\partial x}$, $y\cdot \frac {\partial G}{\partial y}, \frac {\partial G}{\partial x}\cdot x, \frac {\partial G}{\partial y}\cdot x, \frac {\partial G}{\partial x}\cdot y, \frac {\partial G}{\partial y}\cdot y\}$. Since $m$ has no components of degree $2$ then $m\in S'+A(4)+\sum_{i=5}^{\infty }A(i)$. Recall that $m\in A(3)$.
If follows that $m$ is a linear combination of elements from $S''=span _{\K}\{x\cdot \frac {\partial F}{\partial x}$, $x\cdot \frac {\partial F}{\partial y}$, $y\cdot \frac {\partial F}{\partial x}$, $y\cdot \frac {\partial F}{\partial y}, \frac {\partial F}{\partial x}\cdot x, \frac {\partial F}{\partial y}\cdot x, \frac {\partial F}{\partial x}\cdot y, \frac {\partial F}{\partial y}\cdot y\}$.
Therefore $m\in I+A(4)$, and since $I$ is homogeneous $m\in I$, a contradiction.

We now claim that there is a monomial $n\in A(4)$ such that $n\notin J$.
Observe first that if $m\in J\cap A(4)$ then $m\in \K\frac {\partial G}{\partial x}+\K\frac {\partial G}{\partial y}+ S'+S'A(1)+A(1)S'+\sum_{i=5}^{\infty }A(i)$. Recall that $m$ has no terms of degree $2$; hence  $m\in S'+S'A(1)+A(1)S'+\sum_{i=5}^{\infty }A(i)$.
Let $m=m'+m''$ where $m'\in S'$ and $m''\in S'A(1)+A(1)S'+\sum_{i=5}^{\infty }A(i)=I\cap A(4)+\sum_{i=5}^{\infty }A(i)$.

 Observe that since  $m$ has no terms of degree $3$ then $m'$ is a linear combination of elements  $x\cdot \frac {\partial H}{\partial x}-\frac {\partial H}{\partial x}\cdot x+y\cdot \frac {\partial H}{\partial y}-\frac {\partial H}{\partial y}\cdot y$ (this element is zero by Lemma ~\ref{one})
and element $q=\alpha _{1}\cdot (x\cdot {\frac {\partial H}{\partial x}}-{\frac {\partial H}{\partial x}}\cdot x)+ \alpha _{2}( y\cdot {\frac {\partial H}{\partial x}}-{\frac {\partial H}{\partial x}}\cdot y)+\alpha _{3}\cdot ({\frac {\partial H}{\partial y}}\cdot x-x\cdot {\frac {\partial H}{\partial y}})=0$ where $\alpha _{1}, \alpha _{2}, \alpha _{3}$ are as in Lemma ~\ref{one}.

 Therefore $A(4)\cap J= A(4)\cap I+ \K\cdot q$, hence $A(4)\cap J$ has dimension at most $15$, so   there exists a monomial $n\in A(4)$ such that $n\notin J$.

The conclusion:
by the construction any non-trivial linear combination of elements $1$, $x$, $y$, $p(1)$, $p(2)$, $m(1)$, $m(2)$, $n$ is not in $J$, therefore  $\K\langle x,y\rangle  /J$ has at least dimension $8$.

{\bf Case 2.} It is done similarly, with the same notation. In fact it is a bit easier, since  there are at least $3$ monomials $m(1), m(2), m(3)$ of degree $3$ whose nontrivial linear combinations are not in $J+A(4)$, so elements $1, x, y, p(1), p(2), n(1), n(2), n(3)$ and their nontrivial linear combinations are not in $J$.
\end{proof}

\section{ $A_{con}$ and  conjecture of  Wemyss} \label{Dim2}
In this chapter we will consider algebras related to contraction algebras $ A_{con}$.
Over the last decade or so,  a series of new  ideas have been developed on how to describe properties of certain structures in geometry using noncommutative rings such as
reconstruction algebra, MMA and $A_{con}$.
These rings can be described via generators and relations, and they  can be studied using the Gold-Shafarevich theorem and other methods coming from noncommutative ring theory.  Maximal modification algebras (MMA)
 were developed   by Iyama and Wemyss in \cite{IW}. $A_{con}$ is  certain factor of MMA, \cite{WD}.
 Given a commutative 3-dimensional ring R, MMA  is not known to exist in full generality, but exists under mild assumptions.
 Below we give some information about $A_{con}$  provided  by Michael Wemyss  in an e-mail \cite{Wemyss}.
%In the setting of paper of W. Donovan and M.Wemyss, Acon corresponds to noncommutative deformations of curves; the %abelianization of Acon corresponds to classical (commutative) deformations of the curves.

%The general idea of constructing Acon for a commutative ring $R$ is the following: to construct Acon we first construct the
 %Maximal modification algebra (MMa) for $R$. Maximal; modification algebras are of the form $End_R(M)$ for some module $M$; next a quiver algebra is constructed. The verticles of this quiver are
 %some  modules of algebra $A$, namely  the CM modules of algebra $A$.
 %Morphisms between these modules are arrows in the quiver algebra; and if these morphisms satisfy
 %some relations then we add these relations to our  quiver algebra. In this way we obtain the quiver algebra with relations \cite{Wemyss}.

 If R is a 3-dimensional commutative ring with MMA $A$, then if $A$ has finite global dimension, then by a result of Van den Bergh \cite{MB}  it follows that (after completing the ring) the relations of $A$ come from a potential.  This implies that $A_{con}$ comes from a potential as well, \cite{WD, Wemyss}.

 It was shown by Michael Wemyss that the  completion  of the algebra with the potential
$F= x^{2}y+xyx+yx^{2}+y^{3}+y^{4}$ has dimension $8$, and he conjectured that this is the minimal possible dimension.
 We show below that his conjecture is true.
 We notice, that the completion of the algebra with potential $F=x^{3}+y^{3}+(x+y)^{4}$  has dimension $8$ as well.

 We will use the following definition of the completion of an algebra  (for more information on the completion of an algebra see Section $2$ in \cite{DWZ}, page 7, formula (2.3)). Let $\K\langle x,y\rangle$ be a free noncommutative algebra in free generators $x,y$.
We assign to $x$ and to $y$ degree 1, and  denote by $F[n]$  the linear space space spanned by all monomials from $\K\langle x,y\rangle $ whose degree is at least $n$.
 \begin{definition} Let $I$ be an ideal in $\K\langle x,y\rangle $, then the {\it closure} of I is the set of all elements $r \in \K\langle x,y\rangle $ such that
  for every $n$,
$r\in I+F[n]$. If $\bar I$ is the closure of the ideal  $I$, then we will say that $\K\langle x,y\rangle /{\bar I}$ is the  {\it completion}  of the algebra $\K\langle x,y\rangle /I$.
\end{definition}

 \begin{theorem}
 The minimal dimension of the  completion  of a potential algebra is $8$.
\end{theorem}

\begin{proof} Let notation be as above. Observe that  the dimension of  $\K\langle x,y\rangle/I+F[5]$ doesn't exceed the dimension of the
 completion of $\K\langle x,y\rangle/I$, since $I+F[5] \supset \cap (I+F[n])=\bar I$.

 Therefore by the above definition of the closure of an ideal, and by the same proof as in Theorem ~\ref{important}, the  completion of the potential algebra has dimension at least $8$.   On the other hand, Wemyss has shown that there is a potential algebra whose completion has dimension $8$, for example algebra defined by the potential  $F= x^{2}y+xyx+yx^{2}+y^{3}+y^{4}$.
\end{proof}

 We conclude this section with two open questions.

 {\bf Question 1. } Let $A_{F}$ be a potential algebra with potential $F$, which is a sum of terms of degree $4$ or higher. Can $A_{F}$ be finite-dimensional?

 Note that by Theorem ~\ref{deg40}, if  $F$ is  homogeneous then $A_{F}$ is infinite-dimensional.

{\bf Question 2.}  Let $A_{F}$ be a potential algebra whose potential $F$ is a sum of terms of degree $3$ or higher. Can
 $A_{F}$ have dimension $8$?

 Wemyss \cite{Wemyss} conjectured that every potential algebra comes from geometry and is related to $A_{con}$.
 If this were true then, as shown by Toda~\cite{T}, the difference between the dimension of a potential algebra and its abelianization is a linear combination of squares of natural numbers starting from 2, with non-negative integer coefficients. In the next chapter, we consider some special cases of this conjecture.
 Rings which come from geometry have special structure, as shown in \cite{WD}.
In Remark 3.17 in \cite{WD}, Donovan and Wemyss show that every $A_{con}$ has a central element $g$ that cuts to an algebra of special form, in particular the factor algebra $A_{con}/gA_{con}$ has dimension 1, 4, 12, 24, 40 or 60.

 %As explained by Wemyss \cite{Wemyss}: ``  the notation in Remark 3.17 is quite %complicated, and the  equation to look for in the paper \cite{WD}  is (3.C), and the  line below it, which shows that there is a central element g of Acon such that
%$Acon/(g Acon)$ is isomorphic to and algebra denoted $ \Delta_{con}$.
%This $\Delta_{con}$  is one of six explicit algebras (which are not written down %explicitly in the paper, although their presentations are known) that arise from %partial resolutions of Kleinian singularities.  The dimensions of the six possible %algebras are 1, 4, 12, 24, 40 or 60 (this is the content of Remark 3.17).   This means %that for the Acon coming from geometry there is always a central element that cuts to %an algebra of dimension either 1, 4, 12, 24, 40 or 60.''

\section{Difference of dimensions of $A$ and its abelianization via Gopakumar-Vafa invariants }\label{Wem}

In this section we consider the conjecture due to Wemyss,
% particular cases of which are proved in
 \cite{WD}, which says that for finite dimensional algebras the difference between the dimension of a potential algebra and its abelianization is a linear combination of squares of natural numbers starting from 2, with non-negative integer coefficients.
Moreover, in \cite{T} it is shown, that these integer coefficients do coincide with Gopakumar-Vafa invariants \cite{Ka}.

In this section  we prove the conjecture for one example of potential of certain kind, using Gr\"obner basis arguments.

%provide a ton example of a potential algebra of dimension 13, whose abelianization has %dimension 7. Thus the difference of theses dimensions is 6, which can not be written as %a sum of required squares. This disproves the conjecture.

Let $F=x^2y+xyx+yx^2+xy^2+yxy+y^2x+a(y)$, where $a=\sum_{j=3}^n a_jy^j\in\K[y]$ is of degree $n>3$ and has only terms of degree $\geq3$. Let $A$ be the corresponding potential algebra $A=\K\langle x,y\rangle/I$, where the ideal $I$ is generated by $d_xF=xy+yx+y^2$ and $d_yF=xy+yx+x^2+b(y)$ with $b(y)=\sum_{j=3}^n a_jy^{j-1}$. Symbol $B$ stands for the abelianization of $A$: $B=A/Id(xy-yx)$.

Claim 1. $\dim B=n+1$.

\begin{proof}
Clearly $B=\K[x,y]/J$, where $J$ is the ideal generated by $2xy+y^2$ and $2xy+x^2+b(y)$. We use the lexicographical ordering (with $x>y$) on commutative monomials. The leading monomials of the defining relations are $x^2$ and $xy$. Resolving the overlap $x^2y$ completes the commutative Gr\"obner basis of the ideal of relations of $B$ yielding $4yb(y)-3y^3$, which together with defining relations comprise a Gr\"obner basis. The corresponding normal words are $1,x,y,\dots,y^{n-1}$. Hence the dimension of $B$ is $n+1$.
\end{proof}

Claim 2. Denote $c(y)=\frac12(b(y)-b(-y))$ and $d(y)=\frac12(b(y)+b(-y))$, the odd and even parts of $b$. Then $A$ is infinite dimensional if and only if $c=0$ (that is, if and only if $a$ is odd). If $c\neq 0$ and $m=\deg c<n-1=\deg b$, then $\dim A=n+2m-1$. If $c\neq 0$ and $\deg c=\deg b$, then $\dim A=3n-3$. In any case $\dim A-\dim B$ is a multiple of $4$.

\begin{proof} We sketch the idea of the proof. From the defining relation $xy+yx+y^2$ it follows that both $x^2$ and $y^2$ are central in $A$.
The other defining relation $xy+yx+x^2+b(y)$ has the leading monomial $y^{n-1}$ (now we use the deg-lex order on non-commutative monomials assuming $x>y$). One easily sees that if $b$ is even (that is $c=0$), then the defining relations form a Gr\"obner basis. The leading monomials now are $xy$ and $y^{n-1}$, while the normal words are $y^jx^k$ with $0\leq j<n$, $k\geq 0$. Hence $A$ is infinite dimensional.

Assume now that $m=\deg c<n-1=\deg b$. Since $x^2$ and $y^2$ are central, the defining relations imply that so are $xy+yx$ and $c(y)$. In particular, we have a relation $[x,c(y)]=0$. The relation $xy+yx+y^2=0$ allows us to rewrite $[x,c(y)]=0$ as $2c(y)x+c(y)y=0$, providing a relation with the leading monomial $y^mx$. Now, resolving the overlap $y^{n-1}x$, we get a relation with the leading monomial $x^3$. Now one routinely checks, that the defining relations together with the two extra relations we have obtained form a Gr\"obner basis in the ideal of relations. The leading monomials are $xy$, $x^3$, $y^{n-1}$ and $y^mx$. Thus the normal words are $y^j$ with $0\leq j\leq n-2$, $y^jx$ and $y^jx^2$ with $0\leq j\leq m-1$. This gives $\dim A=n+2m-1$ and $\dim A-\dim B=2m-2$, which is a multiple of 4 since $m$ is odd.

Finally, assume that $\deg b<\deg c$. In this case one can verify that the relation $[x,c(y)]=0$ reduces to one with the leading monomial $x^3$ and that the last relation together with the defining relations forms a Gr\"obner basis in the ideal of relations. The leading monomials are $xy$, $x^3$ and $y^{n-1}$. Thus the normal words are $y^j$, $y^jx$ and $y^jx^2$ with $0\leq j\leq  n-2$. This gives $\dim A=3n-3$ and $\dim A-\dim B=2n-4$, which is a multiple of 4 since in this case $n$ is even.
\end{proof}

%\section{A counterexample to the conjecture on the arithmetic structure of the
%difference between dimensions of $A$ and its abelianization }\label{W}
%In this section we consider the following conjecture due to Wemyss, particular cases of
%which are proved in [W,T].
%The conjecture says that  the difference between the dimension of a potential algebra
%and its abelianization is a linear combination of squares of natural numbers starting
%from 2, with non-negative integer coefficients. We provide an example of a potential
%algebra of dimension 13, whose abelianization has dimension 7. Thus the difference of
%theses dimensions is 6, which can not be written as a sum of required squares. This
%disproves the conjecture.
%The following is a counterexample.
%$$F=x^2y \rcirclearrowleft + xy^2 \rcirclearrowleft -y^6,$$
%so the relations are
%$$D_x F= xy+yx+y^2$$
%$$D_y F=x^2+xy+yx-y^5$$
%Dimension of algebra given by these relations $d=13$,
%the dimension of the algebra with added commutator $xy=yx$ is
%$a=7$,
%so the difference is 6, which is not a linear combination of 4, 9.. etc.

 \section{Acknowledgements}
 We would like to thank Michael Wemyss for communicating to us a number of interesting questions on potential algebras and shearing his insights.
 We  are grateful  to the anonymous referees for careful reading and many useful comments, which helped to improve the text.

 This work is funded by the ERC grant 320974.

\vspace{17mm}

\scshape

\noindent   Natalia Iyudu \& Agata Smoktunowicz

\noindent School of Mathematics

\noindent  The University of Edinburgh

\noindent James Clerk Maxwell Building

\noindent The King's Buildings

\noindent Peter Guthrie Tait Road

\noindent Edinburgh

\noindent Scotland EH9 3FD

\noindent E-mail address:\ \ \

{ niyudu@staffmail.ed.ac.uk, a.smoktunowicz@ed.ac.uk }\ \ \

\vskip 11mm

\vskip1truecm

\end{document}